\documentclass[10pt,twoside]{article}
\usepackage{mathrsfs}
\usepackage{}
\usepackage{amsfonts}

\usepackage{amssymb,amsthm,amsmath,hyperref}

\usepackage{graphicx,float}

\numberwithin{equation}{section}

\newtheorem{thm}{Theorem}[section]
\newtheorem{lem}{Lemma}[section]

\newtheorem{prop}{Proposition}[section]

\theoremstyle{definition}
\newtheorem{defn}{Definition}[section]

\theoremstyle{remark}
\newtheorem{rem}{Remark}[section]

\theoremstyle{claim}

\makeatletter \oddsidemargin.9375in \evensidemargin \oddsidemargin
\begin{document}
\thispagestyle{empty}
\setcounter{page}{1}

\noindent


\begin{center}
{\large\bf Well-posedness in Critical Spaces for the
Density-dependent\\ Incompressible Viscoelastic Fluid System}

\vskip.20in

Huazhao Xie and  Yunxia Fu

\end{center}

{\footnotesize
\noindent
{\bf Abstract.}
 We are concerned with  the well-posedness  of
 the density-dependent incompressible viscoelastic  fluid system. By Schauder-Tychonoff fixed point argument,
when $\|{1}/{\rho_0}-1\|_{\dot{B}_{p,1}^{{N}/{p}}}$ is small, local
well-posedness is showed to hold in Besov space. Furthermore,
provided the initial data $(1/\rho_0 -1,\upsilon_0,U_0-I)$ is small
under certain
 norm, we also get the  existence of the global solution.\\[3pt]
 }
\vskip.2in

\section{Introduction}

\noindent
Viscoelastic fluids have a wide range of applications and hence have
received a great deal of interest. Examples and applications of
viscoelastic fluids include  oil, liquid polymers, mucus, liquid
soap, toothpaste, clay, ceramics, coatings, drug delivery systems
for controlled drug release,  viscoelastic blood flow  past valves
and so on, see [7] for more applications. The motion of a
density-dependent incompressible viscoelastic fluid is described by
the following inhomogeneous Oldroyd system:
 \begin{equation}
    \left\{
     \renewcommand{\arraystretch}{1.25}
    \begin{array}{l}
        \rho_t +  $div$( \rho\upsilon)=0, \quad (t,x)\in (0,+\infty)\times {\mathbb{R}}^N,\\
 (\upsilon\upsilon)_t + $div$\ (\rho\upsilon \otimes\upsilon )+ \nabla
P=\mu \Delta \upsilon +$div$\ (\rho UU^T) ,\\
 U_t +\upsilon\cdot \nabla U=\nabla \upsilon U,\\
 $div$\ \upsilon=0
    \end{array}
    \right.\label{bl-E1.1}
    \end{equation}
supplemented  with the initial data
\begin{equation}\label{intial}
(\rho,\upsilon, U)|_{t=0}=(\rho_0,\upsilon_0,U_0),
\end{equation}
where $N\geq 2, \ \rho(t,x), \upsilon(t,x), P(t,x)$ and
$U(t,x)=(U^{ij}(t,x))_{N\times N}$ denote the density, velocity,
hydrodynamic pressure and the deformation tensor respectively, the
  viscosity coefficient $\mu>0$ is a constant.

  In the context of hydrodynamics, the
motion of the fluid flow is descried by the particle trajectory
$x(t,X)$, where material points $X$ are deformed to the spatial
position $x(t,X)$ at the time $t$. The deformation tensor
$\widetilde{U}(t,x)=\frac{\partial x}{\partial X}(t,X)$, when we
work in Eulerian coordinate, we denote it by
$U(t,x)=\widetilde{U}(t,X^{-1}(t,x))$. Applying the chain rule, we
see that $U(t,x)$ satisfies the transport equation
(\ref{bl-E1.1})$_3$, which stands for
$$
U_t^{ij}+\upsilon^k \nabla_k U^{ij}=\nabla_k \upsilon^i U^{kj},\quad
\text{for}\quad i,j=1,\cdots,N,
$$
where $ \nabla_i =\frac{\partial}{\partial x_i},\
U^{ij}=\frac{\partial x^i}{
\partial X^j}, \ \nabla_j \upsilon^i=(\nabla \upsilon)_{ij}. $

We assume that the initial data  satisfy the constrains
\begin{equation}
    \left\{
     \renewcommand{\arraystretch}{1.25}
    \begin{array}{l}
 $div$\ \upsilon_0 =0,\\
 $div$\ (\rho_0 U_0^T) =0,\\
 U_0^{lk}\nabla_l U_0^{ij}- U_0^{ij}\nabla_l U_0^{ik}=0.
    \end{array}
    \right.\label{bl-E1.2}
    \end{equation}
Using (\ref{bl-E1.2}), it is easy to obtain that   div $(\rho
U^T)=0$ and
\begin{equation}\label{bl-E1.3}
U^{lk}\nabla_l U^{ij}-U^{lj}\nabla_l U^{ik}=0,
\end{equation}
hold for all time, see [8,12].

System (\ref{bl-E1.1}) has been studied extensively. When the
density is a constant, system (\ref{bl-E1.1}) governs the
homogeneous incompressible viscoelastic fluids, and there exist rich
results in the literature for the global existence of classical
solutions, see [2,9,10,11] and the references therein.
Let $H=U-I$ be the perturbation of deformation tensor $U$, Lei et
al. [9] find that
\begin{equation}\label{bl-E1.4}
\nabla_k H^{ij}-\nabla_j H^{ik}= H^{lj}\nabla_l H^{ik} -H^{lk}
\nabla_l H^{ij},
\end{equation}
which is useful to prove the global existence. When density
$\rho(t,x)$ is not a constant, the problem related to existence
becomes more complicated and not much has been done.  Qian and Zhang
[12]  got the well-posedness in critical spaces for the
compressible viscoelastic fluids. One of the difficulties in proving
the global existence is the lacking of dissipative estimate for the
deformation gradient, to overcome this difficulty, Lei et al. in
[9] introduced an auxiliary function, the authors in
[12] explore the smoothing effect and damping effect of the
system by viewing it a hyperbolic-parabolic coupled system. However,
This paper is devoted to the density-dependent incompressible
viscoelastic fluids, those methods do not apply directly. We need to
estimate  the pressure term. By using standard energy estimates for
dyadic blocks, we can get the estimate of $\nabla P$. Hu and Wang in [13]  considered the
the three-dimensional density-dependent incompressible viscoelastic fluids, they got the
existence and uniqueness of the global strong solution with small initial data. Our assumptions and methods are different, we consider
the problem in critical spaces.

We shall use scaling considerations to find which spaces are
critical for (\ref{bl-E1.1}).  System (\ref{bl-E1.1}) is invariant
by the rescaling $(\rho,\upsilon,P,U)\mapsto
(\rho_l,\upsilon_l,P_l,U_l)$ with
\begin{eqnarray*}
&& \rho_l (t,x)=\rho(l^2 t,lx),\quad \upsilon_l (t,x)= l \upsilon(l^2
t,lx),\\
&& P_l (t,x)=l^2 P(l^2 t,lx),\quad U_l (t,x)=l^2 U(l^2
t,lx).
\end{eqnarray*}
This motivates the following definition:

\begin{defn}
We say that a functional space is critical with respect to the
scaling of the equations if the associated norm is invariant under
the transformation: $(\rho,\upsilon,P,U)\mapsto
(\rho_l,\upsilon_l,P_l,U_l)$ (up to a constant independent of $l$).
\end{defn}

In Sobolev spaces setting, the above definition would lead us to
consider initial data in $\dot{H}^{N/p}\times
(\dot{H}^{N/p-1})^N\times (\dot{H}^{N/p})^{N^2}$. If $\rho$ vanishes
or becomes unbounded, system (\ref{bl-E1.1}) degenerates, it is
reasonable to  assume that $\rho_0 \in L^{\infty}$. For technical
reasons, we suppose that the initial data belong to a somewhat
smaller homogeneous Besov space $ \dot{B}_{p,1}^{{N}/{p}} \times (
\dot{B}_{p,1}^{{N}/{p}-1} )^{N} \times ( \dot{B}_{p,1}^{{N}/{p}}
)^{N^2}. $ Set
$$
\sigma=\frac{1}{\rho}-1,\quad H=U-I,
$$
the system (\ref{bl-E1.1})-(\ref{intial}) can be reformulated as
follows
\begin{equation}
    \left\{
     \renewcommand{\arraystretch}{1.25}
    \begin{array}{l}
        \sigma_t +  \upsilon\cdot \nabla \sigma=0, \\
 \upsilon^i_t + \upsilon \cdot\nabla \upsilon^i+ (\sigma+1)\nabla_i
P=\mu (\sigma +1)\Delta \upsilon^i +\partial_k H^{ik} +H^{jk}\partial_j H^{ik},\\
 H_t +\upsilon\cdot \nabla H=\nabla \upsilon (H+I),\\
 $div$\ \upsilon=0, \\
 (\sigma,\upsilon, H)|_{t=0}=(\sigma_0,\upsilon_0,H_0).
    \end{array}
    \right.\label{bl-E3.1}
    \end{equation}
    The well-posedness of the system (\ref{bl-E1.1})-(\ref{intial})
    is equivalent to the system (\ref{bl-E3.1}).

 The main results of this paper are as follows.
\begin{thm}\label{th1}  For the system
(\ref{bl-E3.1}), let $p\in[1,N]$, there exists a constant $c=c(N)$,
such that for any   $\upsilon_0 \in (\dot{B}_{p,1}^{{N}/{p}-1})^N$
with div $\upsilon_0 =0$, $H_0 \in (\dot{B}_{p,1}^{N/p})^{N^2}$ and
$\sigma_0 \in \dot{B}_{p,1}^{{N}/{p}}$ with
$\|\sigma_0\|_{\dot{B}_{p,1}^{N/p}}\leq c$, then there exists a
positive time $T\in (0,+\infty)$ such that system (\ref{bl-E3.1})
has a unique solution $(\sigma,\upsilon, H,\nabla P)$ with
\begin{eqnarray*}
\sigma\in{\mathcal{C}}([0,T];\dot{B}_{p,1}^{s}),&& \upsilon\in
\Big({\mathcal{C}}([0,T];\dot{B}_{p,1}^{s-1})\cap L^1
([0,T];\dot{B}_{p,1}^{s+1})\Big)^{N},\\
 H \in {\mathcal{C}}([0,T];\dot{B}_{p,1}^{s})^{N^2},
&\text{and}& \nabla P \in (L^1 ([0,T];\dot{B}_{p,1}^{s-1})^N.
\end{eqnarray*}
In addition, for $p=2$, we denote $B^s:=B^s_{2,1}$,  there exists a
constant $\eta$ such that, if
$$
\|\sigma_0\|_{B^{\frac{N}{2}}\cap B^{\frac{N}{2}-1}}+ \|H_0
\|_{B^{\frac{N}{2}}\cap B^{\frac{N}{2}-1}} +\|\upsilon_0 \|_{
B^{\frac{N}{2}-1}}\leq \eta,
$$
then the system (\ref{bl-E3.1})  has a unique global solution with
$(\sigma,\upsilon, H)\in {\mathcal{H}}^{N/2}$, where
\begin{eqnarray*}
{\mathcal{H}}^s &:=&\big( L^2(R^+; B^s)\cap {\mathcal {C}}(R^+; B^s
\cap B^{s-1}) \big)\\
&& \quad \times \big( L^1(R^+; B^{s+1})\cap {\mathcal
{C}}(R^+; B^s \cap B^{s-1}) \big)^{N}\\
&& \quad \times \big( L^2(R^+; B^s)\cap {\mathcal {C}}(R^+; B^s \cap
B^{s-1}) \big)^{N^2},
\end{eqnarray*}
and $\nabla P \in  L^1(R^+; B^{s-1})^{N}$.
\end{thm}

This paper is structured as follows. Section 2 is devoted to  recall
some basic results on Besov spaces. In Section 3, using
Schauder-Tychonoff fixed point argument, we prove the existence and
uniqueness of the local solution. At last, we concentrate on the
proof of the global existence of solution.

Notation: Throughout the paper, $C$ stands for a universal constant.
$\mathcal{Z}'({\mathbb{R}}^N)$ stands for the dual space
${\mathcal{Z}}({\mathbb{R}}^N)=\{ f\in \varphi({\mathbb{R}}^N):
D^{\alpha}\hat{f}(0)=0, \forall \alpha \in {\mathbb{N}}^N \
\text{multi-index}\}$. The notation $L_T^p(X)$ stands for the set of
measurable functions on $(0,T)$ with values in $X$, such that
$\|\cdot\|_X \in L^p (0,T)$, where $X$ be a Banach space,
$p\in[1,+\infty]$.

\section{Basic results on Besov spaces}\label{bl-S2}

\noindent
In this section, we mainly review some results on Besov spaces. At
first, we introduce the Littlewood-Paley theory. The homogeneous
Littlewood-Paley decomposition relies upon a dyadic partition of
unity: radial function $\varphi\in \varphi({\mathbb{R}}^N)$
supported in the shell ${\mathcal {C}}:=\{ \xi\in {\mathbb{R}}^N \ |
\ \frac{3}{4}\leq \xi\leq \frac{8}{3}\}$  such that
$$
\sum_{q\in \mathbb{Z}}\varphi(2^{-q}\xi)=1, \ \text{for all} \
\xi\neq 0.
$$
The homogeneous dyadic blocks and low frequency cut-off are defined by
$$
\Delta_q f:=\varphi(2^{-q}D)f,\quad\text{and}\quad S_q f:=\sum_{k\leq q-1}\Delta_k f \quad \text{for}\quad q\in\mathbb{Z}.
$$
It is easy to verify the following properties hold:
$$
\Delta_q \Delta_k f \equiv 0 \ \ \text{for}\ \ |q-k|\geq 2;\quad
\Delta_q (S_{k-1}f \Delta_k f) \equiv 0 \ \ \text{for}\ \ |q-k|\geq
5.
$$
The definition of the Besov space depend on the Littlewood-Paley
decomposition.

\begin{defn}\label{d2.1}
For $s\in \mathbb{R}$,  $ (p,r) \in[1,+\infty]^2$ and $ u\in
{\mathcal {Z}}'({\mathbb{R}}^N)$, we set
$$
\|u\|_{\dot{B}_{p,r}^s}:=\big{\|} 2^{js}\|\Delta_j u\|_{p} \big{\|}_{l^r},
$$
with the usual change if $r=+ \infty$. The homogeneous
Besov space $\dot{B}_{p,r}^s$ is defined by
$$
\dot{B}_{p,r}^s :=\big\{  u\in {\mathcal {Z}}'({\mathbb{R}}^N):
\|u\|_{\dot{B}_{p,r}^s} <+\infty      \big\}.
$$
\end{defn}
Next, we introduce the Besov-Chemin-Lerner space $\tilde{L}^q_T
(\dot{B}_{p,r}^s)$ which is initiated in [1].

\begin{defn}\label{d2.2}
For $s\in \mathbb{R}$,  $ (p,r,k) \in[1,+\infty]^3$ and $ T \in (0, +\infty]$.
The space $\tilde{L}_T^k (\dot{B}_{p,r}^s)$ is defined by
$$
\tilde{L}_T^k (\dot{B}_{p,r}^s):= \big\{  u\in L^k (0,T;{\mathcal
{Z}}'({\mathbb{R}}^N): \|u\|_{ \tilde{L}_T^k (\dot{B}_{p,r}^s)}
<+\infty      \big\},
$$
where
$$
\|u\|_{\tilde{L}_T^k (\dot{B}_{p,r}^s)}:=\Big{\{}\sum_{q\in
\mathbb{Z}} 2^{qrs}\Big( \int_0^T \|\Delta_q u(t)\|_{p}^k
dt\Big)^{\frac{r}{k}} \Big{\}}^{\frac{1}{r}},
$$
with the usual change if $r=+ \infty$.
\end{defn}
By virtue of Minkowski's inequality, we have
\begin{eqnarray*}
\|u\|_{\tilde{L}_T^k (\dot{B}_{p,r}^s)}&\leq& \|u\|_{L_T^k (\dot{B}_{p,r}^s)} \quad \text{if}\quad k\leq r,\\
\|u\|_{\tilde{L}_T^k (\dot{B}_{p,r}^s)}&\geq& \|u\|_{L_T^k (\dot{B}_{p,r}^s)} \quad \text{if}\quad k\geq r.
\end{eqnarray*}
In order to get the global existence result, we need to give the
definition of hybrid Besov spaces.
\begin{defn}\label{d4.1}
For $\mu >0, r \in [1,+\infty]$ and $s\in \mathbb{R}$, we define the
hybrid Besov space $\tilde{B}_{\mu}^{s,r}$
 as the set of functions $u$ such that
$$
\|u\|_{\tilde{B}_{\mu}^{s,r}}:=\sum_{q\in \mathbb{Z}} 2^{qs}\max(
\mu,2^{-q})^{1-\frac{2}{r}}\|\Delta_q u\|_{L^2}<\infty.
$$
\end{defn}
Let us  list some important properties of  the Besov spaces, see
[3].

\begin{lem}\label{l-2.1}
For $s\in \mathbb{R}, \ (p,r)\in [ 1, +\infty]$. The following
inequalities hold
true:\\
(i) there exists a universal constant $C$ such that
\begin{equation}\label{e2.0}
C^{-1}\|u\|_{\dot{B}_{p,r}^s}\leq \|\nabla u\|_{\dot{B}_{p,r}^{s-1}}\leq C \|u\|_{\dot{B}_{p,r}^s};
\end{equation}
(ii) if $s_1 , s_2 \leq \frac{N}{p}$ and $s_1+ s_2 > N \max\big( 0,
\frac{2}{p}-1 \big) $,
\begin{equation}\label{e2.1}
\|uv\|_{\dot{B}_{p,1}^{s_1 +s_2 -\frac{N}{p}}}\leq C
\|u\|_{\dot{B}_{p,1}^{s_1}}\|v\|_{\dot{B}_{p,1}^{s_2}};
\end{equation}
(iii) if $s_1 \leq \frac{N}{p}, s_2 < \frac{N}{p}$ and $s_1+ s_2 \geq
N \max\big( 0, \frac{2}{p}-1 \big) $,
\begin{equation}\label{e2.2}
\|uv\|_{\dot{B}_{p,\infty }^{s_1 +s_2 -\frac{N}{p}}}\leq C
\|u\|_{\dot{B}_{p,1}^{s_1}}\|v\|_{\dot{B}_{p,\infty}^{s_2}}.
\end{equation}
\end{lem}
\begin{lem}\label{l-2.2}
Let $1\leq k, p,q,q_1,q_2\leq \infty$ with
$\frac{1}{q_1}+\frac{1}{q_2}=\frac{1}{q}$, then\\
(i) if $f\in \tilde{L}_T^{q_1}(\dot{B}_{p,1}^{s_1}), g\in
\tilde{L}_T^{q_2}(\dot{B}_{p,1}^{s_2})$ and  $s_1 , s_2 \leq
\frac{N}{p}$, $s_1+ s_2
> N \max\big( 0, \frac{2}{p}-1 \big) $,
\begin{equation}\label{e2.3}
\|fg\|_{\tilde{L}_T^{q}(\dot{B}_{p,1}^{s_1 +s_2-\frac{N}{p}})}\leq C
\|f\|_{\tilde{L}_T^{q_1}(\dot{B}_{p,1}^{s_1})}\|g\|_{\tilde{L}_T^{q_2}(\dot{B}_{p,1}^{s_2})};
\end{equation}
(ii) if $f\in \tilde{L}_T^{q_1}(\dot{B}_{p,1}^{s_1}), g\in
\tilde{L}_T^{q_2}(\dot{B}_{p,\infty}^{s_2})$ and $s_1 \leq
\frac{N}{p}, s_2 < \frac{N}{p}$, $s_1+ s_2 \geq N \max\big( 0,
\frac{2}{p}-1 \big) $,
\begin{equation}\label{e2.4}
\|fg\|_{\tilde{L}_T^{q}(\dot{B}_{p,\infty }^{s_1 +s_2
-\frac{N}{p}})}\leq C
\|f\|_{\tilde{L}_T^{q_1}(\dot{B}_{p,1}^{s_1})}\|g\|_{\tilde{L}_T^{q_2}(\dot{B}_{p,\infty}^{s_2})}.
\end{equation}
(iii) for $f\in \tilde{L}_T^{k}(\dot{B}_{p,1}^{s}), s\in \mathbb{R}$ and $\epsilon \in (0,1]$, we have
\begin{equation}\label{e2.5}
\|f\|_{\tilde{L}_T^{k}(\dot{B}_{p,1 }^{s})}\leq C
\frac{\|f\|_{\tilde{L}_T^{k}(\dot{B}_{p,\infty}^{s})}}{\epsilon}\log
\Big( e +
\frac{\|f\|_{\tilde{L}_T^{k}(\dot{B}_{p,\infty}^{s-\epsilon})}+
\|f\|_{\tilde{L}_T^{k}(\dot{B}_{p,\infty}^{s+\epsilon})}}{\|f\|_{\tilde{L}_T^{k}(\dot{B}_{p,\infty}^{s})}}\Big).
\end{equation}

\end{lem}

\begin{lem}\label{a.3}(see [4])
Let $0<R_1<R_2$ and $\psi\in {\mathcal {C}}_0^\infty
({\mathbb{R}}^N)$ be supported in the annulus $C(0,R_1,R_2). $ For
all indices $s,t,p,\alpha_1,\alpha_2$ and $\alpha_2'$ such that
$1\leq p, \alpha_1,\alpha_2,\alpha_2' \leq +\infty, \ 1/\alpha_2 +
1/ \alpha_2' =1/ \alpha_1, \ t\leq N/ p +1. \ 1\leq s \leq N/p +1$
and $s+t>1$, there exists $(c_q)_{q\in \mathbb{Z}}$ such that
$\sum_{q} c_q \leq 1$ and
\begin{eqnarray}\label{j}
&& \|\text{div} [A,\psi(2^{-q}D)]\nabla B\|_{L_T^{\alpha_1}(L^p)} \nonumber\\
&& \leq
Cc_q 2^{-q(s+t-2-N/p)}\|\nabla
A\|_{\tilde{L}_T^{\alpha_2'}(\dot{B}_p^{s-1})}\|\nabla
B\|_{\tilde{L}_T^{\alpha_2}(\dot{B}_p^{t-1})}
\end{eqnarray}
\end{lem}
Now we prove a useful estimate.

\begin{lem}\label{12.3}
Let $1<p\leq N, 1\leq \alpha \leq +\infty$, and $u$ be a solution of
\begin{equation}\label{ty}
    -\text{div} (a\nabla u)=f  ,
    \end{equation}
  where  the diffusion coefficient $a(t,x)\in \tilde{L}_T^\infty
  (\dot{B}_{p,1}^{\frac{N}{p}})$ and
 $ 0<\underline{a}<a(t,x)<\overline{a}.$ Then  the following estimate holds
   \begin{equation}\label{le}
\|\nabla u \|_{\tilde{L}_T^\alpha
(\dot{B}_{p,1}^{\frac{N}{p}-1})}\leq C_p \| f \|_{\tilde{L}_T^\alpha
(\dot{B}_{p,1}^{\frac{N}{p}-2})}+
 C_p \|\nabla a \|_{\tilde{L}_T^\infty (\dot{B}_{p,1}^{\frac{N}{p}-1})}
 \|\nabla u \|_{\tilde{L}_T^\alpha (\dot{B}_{p,1}^{\frac{N}{p}-1})}.
\end{equation}
\end{lem}
\begin{proof}
The proof of this lemma lies on standard energy estimates for dyadic
blocks.

Applying $\Delta_j$ to (\ref{ty}) and denoting $R_j := -\text{div}
([a,\Delta_j]\nabla u)$,  we obtain
$$
-\text{div}(a\nabla \Delta_j u)=\Delta_j f + R_j.
$$
Multiplying both sides of the above equation by  $|\Delta_j u|^{p-2}
\Delta_j u$ and integrating over ${\mathbb{R}}^N$, we have
$$
-\int_{{\mathbb{R}}^N} \text{div}(a\nabla \Delta_j u) | \Delta_j
u|^{p-2}\Delta_j u dx =\int_{{\mathbb{R}}^N }(\Delta_j f + R_j)|
\Delta_j u|^{p-2}\Delta_j u dx.
$$
From [4], we have the following inequality
$$
C\underline{a}\big( \frac{p-1}{p^2} \big) 2^{2j}\|\Delta_j
u\|_{L^p}^p \leq -\int_{{\mathbb{R}}^N} \text{div}(a\nabla \Delta_j
u) | \Delta_j u|^{p-2}\Delta_j u dx,
$$
using the above inequality and H$\ddot{o}$lder inequality,  we get
$$
C_p \underline{a} 2^{2j}\|\Delta_j u\|_{L^p}^p \leq \|\Delta_j
u\|_{L^p}^{p-1}(\|R_j\|_{L^p}+\|\Delta_j f\|_{L^p}).
$$
By integration with respect to $t$ over $(0,T)$,  we have
\begin{eqnarray*}
C_p \|\nabla u \|_{\tilde{L}_T^\alpha
(\dot{B}_{p,1}^{\frac{N}{p}-1})} \leq \| f\|_{\tilde{L}_T^\alpha
(\dot{B}_{p,1}^{\frac{N}{p}-2})}+
  \|\nabla a \|_{\tilde{L}_T^\infty (\dot{B}_{p,1}^{\frac{N}{p}-1})}
 \|\nabla u \|_{\tilde{L}_T^\alpha (\dot{B}_{p,1}^{\frac{N}{p}-1})},
\end{eqnarray*}
here we have  used (\ref{e2.0}) and (\ref{j}).
\end{proof}

\begin{rem}\label{rk1} We also have an estimate for solutions of
(\ref{ty}) in $\tilde{L}_T^\alpha
(\dot{B}_{p,\infty}^{\frac{N}{p}})$, with the same assumptions as in
Lemma \ref{12.3}, the following inequality holds:
 \begin{equation}\label{les}
\|\nabla u \|_{\tilde{L}_T^\alpha (\dot{B}_{p,\infty}^{-1})}\leq C_p
\| f \|_{\tilde{L}_T^\alpha (\dot{B}_{p,\infty}^{-2})}+
 C_p \|\nabla a \|_{\tilde{L}_T^\infty (\dot{B}_{p,\infty}^{0})}
 \|\nabla u \|_{\tilde{L}_T^\alpha (\dot{B}_{p,\infty}^{-1})}.
\end{equation}
\end{rem}
The proof goes along the lines of the proof of Lemma \ref{12.3}, and
use  the following enequality which can be found in [5],
$$
\sup_{j}2^{-2j}\|\text{div}\ ([a,\Delta_j]\cdot \nabla u)\|_{L^p}
\leq C\sup_{j}2^{-2j}\| \nabla a\|_{\dot{B}_{p,\infty}^0}\| \nabla
u\|_{\dot{B}_{p,\infty}^{-1}}.
$$
We state the classical estimates in Besov space for the transport
and heat equations, see [3].

\begin{prop}\label{p2.1}
Let  $s\in \big( -1-N\min(\frac{1}{p},\frac{1}{p'}),1+\frac{N}{p}
\big)$, and $1 \leq p,r \leq+\infty $, and $s=1+ \frac{N}{p}$, if
$r=1 $. Let $\upsilon$ be a solenoidal vector such that $\nabla
\upsilon \in L_T^1 (\dot{B}_{p,r}^{\frac{N}{p}}\cap L^{\infty})$.
Assume that $u_0 \in \dot{B}_{p,r}^{s}, \ g\in L_T^1
(\dot{B}_{p,r}^{s})$, and $f$ solves
 \begin{equation}
    \left\{
     \renewcommand{\arraystretch}{1.25}
    \begin{array}{l}
        \partial_t u+  \upsilon\cdot \nabla u=g, \\
 u|_{t=0} = u_0 .
    \end{array}
    \right.\label{e-2.3}
    \end{equation}
 Then for any $t\in [0,T] $, we have
$$
\|u\|_{\tilde{L}_T^{\infty} (\dot{B}_{p,r}^s)}\leq e^{CV(t)}
\Big{(}\|u_0 \|_{ \dot{B}_{p,r}^s} +  \int_0^t
e^{-CV(\tau)}\|g(\tau)\|_{\dot{B}_{p,r}^s} d \tau
 \Big{)},
$$
where  $V(t):=\int_0^t \|\nabla
\upsilon(\tau)\|_{\dot{B}_{p,r}^{{N}/{p}}\cap L^{\infty}} d\tau$. If
$r<+\infty$, then $u\in {\mathcal
{C}}([0,T];\dot{B}_{p,r}^s)$.
\end{prop}

\begin{prop}\label{p2.2}
Let $s\in\mathbb{R}$  and $1\leq k,p,r \leq +\infty$. Assume that
$u_0 \in \dot{B}_{p,r}^s , f\in \tilde{L}_T^{k}
(\dot{B}_{p,r}^{s-2+\frac{2}{p}})$,  and $u$ solves
\begin{equation}
    \left\{
     \renewcommand{\arraystretch}{1.25}
    \begin{array}{l}
        \partial_t u -  \mu \Delta u=f, \\
 u|_{t=0} = u_0 .
    \end{array}
    \right.\label{e-2.4}
    \end{equation}
Denote $1/k_2 =1+1/k_1 -1/k$. Then there exist  positive constants
$c$ and $C=C(N)$ such that for all $k_1 \in [k,+\infty] $, we have
\begin{eqnarray*}
\|u\|_{\tilde{L}_T^{k_1} (\dot{B}_{p,1}^{s+2/ k_1})}&\leq& C
\Big{\{}\sum_{q\in {\mathbb{Z}}}2^{qs}\|\Delta_q u_0 \|_{L^p } \Big(
\frac{1-e^{-c\mu T 2^{2q}k_1}}{c\mu k_1} \Big)^{\frac{1}{k_1}}\\
 &&+\sum_{q\in {\mathbb{Z}}}2^{q(s-2+2/k)}\|\Delta_q f \|_{L^k_T (L^p) }
\Big( \frac{1-e^{-c\mu T 2^{2q}k_2}}{c\mu k_2} \Big)^{\frac{1}{k_2}}
\Big{\}}.
\end{eqnarray*}
Moreover, there holds
$$
\mu^{\frac{1}{k_1}}\|u\|_{\tilde{L}_T^{k_1} (\dot{B}_{p,r}^{s+2/
k_1})}\leq C\Big( \|u_0\|_{ \dot{B}_{p,r}^{s}} + \mu^{\frac{1}{k}
-1}\|f \|_{\tilde{L}_T^{k} (\dot{B}_{p,r}^{s-2 +2/ k})}\Big).
$$
If $r<+\infty$, then $u$ belongs to  ${\mathcal
{C}}([0,T];\dot{B}_{p,r}^s)$.
\end{prop}

We also give a  estimate  for the linear hyperbolic and parabolic
coupled system
\begin{equation}
    \left\{
     \renewcommand{\arraystretch}{1.25}
    \begin{array}{l}
        c_t +  \upsilon \cdot \nabla c+\Lambda d=f, \\
 d_t + \upsilon\cdot \nabla d -\mu \Delta d -\Lambda c=g,
    \end{array}
    \right.\label{e4.1}
    \end{equation}
where $\Lambda=(-\Delta)^{1/2}$.
\begin{prop}\label{p4.1}(see [6])
Let $(c,d)$ be a solution of (\ref{e4.1}) on $[0,T)$ with initial
data $(c_0,d_0) , 1-N/2< s\leq 1+N/2$ and $V(t)=\int_0^t \|v
\|_{B^{\frac{N}{2} +1}} d\tau$. The following estimate holds on
$[0,T)$
\begin{eqnarray*}
&& \|c(t)\|_{\tilde{B}_{\mu}^{s,\infty}}+\|d(t)\|_{B^{s-1}}+\mu
\int_0^t (\|c(\tau)\|_{\tilde{B}_{\mu}^{s,1}}+
        \|d(\tau)\|_{B^{s+1}})d\tau\\
&&\quad \leq C e^{CV(t)}\Big(
\|c_0\|_{\tilde{B}_{\mu}^{s,\infty}}+\|d_0\|_{B^{s-1}}\\
&& \quad \quad+ \int_0^t
e^{-CV(\tau)}
    (\|f(\tau)\|_{\tilde{B}_{\mu}^{s,\infty}}
     +\|g(\tau)\|_{B^{s-1}})d\tau \Big).
\end{eqnarray*}
where $C=C(N,s)$.
\end{prop}


\section{Well-posedness in critical spaces}\label{bl-S3}

\noindent
For the system (\ref{bl-E3.1}), when the initial density is small in
$\dot{B}_{p,1}^{N/p}$, the local well-posedness can be obtained by
mean of the following form of Schauder-Tychonoff fixed point
argument. Furthermore, when  initial data $(\sigma_0,\upsilon_0,
H_0)$ is small under certain norm, we can obtain the global
existence result.

\begin{thm}\label{th3.1}
({\bf Hukuhara}) Let $K$ be a convex subset of a locally convex
topological linear space $E$, and $\Phi$ be a continuous
self-mapping of $K$. If $\Phi(K)$ is contained in a compact subset
of $K$, then $\Phi$ has a fixed point in $K$.
\end{thm}

Let  us briefly enumerate the main steps of the proof: In the first
step, we show  the local existence problem amounts to find a fixed
point for some map $\Phi$.  In the next two steps, we state various
$\textit{a priori}$ estimates for $\Phi$. In the fourth step, we
show that Hukuhara's theorem indeed applies. Step five is devoted to
the uniqueness. At last, when the initial data is small under
certain norm, we give the proof of global existence.

Step 1. Construction of the functional $\Phi$\\
 Define $\Phi$  by
$(\sigma,\upsilon,H)=\Phi(a,u,\xi)$, where $(\sigma,\upsilon,H)$ is
the solution of the linear problem
\begin{equation}
    \left\{
     \renewcommand{\arraystretch}{1.25}
    \begin{array}{l}
        \sigma_t +  u \cdot \nabla \sigma=0, \\
 \upsilon_t -\mu \Delta \upsilon=G- (a+1)\nabla P,\\
 H_t +u\cdot \nabla H=\nabla u (\xi+I),\\
 $div$ \upsilon=0, \\
 (\sigma,\upsilon, H)|_{t=0}=(\sigma_0,\upsilon_0,H_0),
    \end{array}
    \right.\label{bl-E3.2}
    \end{equation}
with $\sigma=\sigma(a)$ and
$$
G:= -u \cdot\nabla u  +\mu a \Delta u +\text{div} \xi +\xi^T \cdot
\nabla\xi .
$$

Step 2. $\textit{A priori}$ estimates

We shall prove that for suitably small $T$ and  $\|\sigma_0
\|_{\dot{B}_{p,1}^{{N}/{p}}}$, the functional $\Phi$ has a fixed
point in the Banach space
$$
\tilde{E}_T^p :=\tilde{L}_T^\infty (\dot{B}_{p,1}^{\frac{N}{p}}) \times
\big( \tilde{L}_T^\infty (\dot{B}_{p,1}^{\frac{N}{p}-1})\cap \tilde{L}_T^1 (\dot{B}_{p,1}^{\frac{N}{p}+1}) \big)^{N}
 \times \tilde{L}_T^\infty (\dot{B}_{p,1}^{\frac{N}{p}})^{N^2}.
$$
Let $E_0 := \|\upsilon_0 \|_{\dot{B}_{p,1}^{\frac{N}{p}-1}}+\|H_0 \|_{\dot{B}_{p,1}^{\frac{N}{p}}}$ and
 $\|\sigma_0 \|_{\dot{B}_{p,1}^{\frac{N}{p}}}=R_0$.  $C_0>0$ and $(R, \eta)\in(0,1)^2 $ to be fixed hereafter,
 we denote
\begin{eqnarray*}
{\mathcal {A}}=\Big\{ (\sigma,\upsilon,H) \in \tilde{E}_T^p :  &&
\|\sigma \|_{\tilde{L}_T^\infty(\dot{B}_{p,1}^{\frac{N}{p}})}\leq
R,\quad
   \|\upsilon \|_{\tilde{L}_T^1 (\dot{B}_{p,1}^{\frac{N}{p}+1})}
   +\|\upsilon \|_{\tilde{L}_T^2 (\dot{B}_{p,1}^{\frac{N}{p}})}\leq \eta,\\
&&\|\upsilon \|_{\tilde{L}_T^\infty(\dot{B}_{p,1}^{\frac{N}{p}-1})}
   +\|H \|_{\tilde{L}_T^\infty(\dot{B}_{p,1}^{\frac{N}{p}})}\leq C_0 E_0. \Big\}
\end{eqnarray*}
 We claim that if $T,R, R_0$ and $\eta$ are small
enough, $\Phi$ maps $\mathcal {A}$ to $\mathcal {A}$.

In what follows, we assume $(a,u,\xi)\in {\mathcal {A}}$ and denote
$s:=\frac{N}{p},\ U(t):=\int_0^t \|\nabla u(\tau)
\|_{\dot{B}_{p,1}^{\frac{N}{p}}}d \tau$ for convenience. Using
Proposition \ref{p2.1}, we have
\begin{equation}\label{sig}
\|\sigma \|_{\tilde{L}_T^\infty(\dot{B}_{p,1}^{s})}\leq e^{CU(T)}\|
\sigma_0 \|_{\dot{B}_{p,1}^{s}}\leq e^{C\eta}R_0,
\end{equation}
and
\begin{eqnarray}\label{bl-E3.3}
\renewcommand{\arraystretch}{1.25}
\|H \|_{\tilde{L}_T^\infty(\dot{B}_{p,1}^{s})} &\leq& e^{CU(T)}\big(
\| H_0 \|_{\dot{B}_{p,1}^{s}}
  +\|\nabla u (I+\xi)\|_{\tilde{L}_T^1(\dot{B}_{p,1}^{s})} \big)\nonumber\\
&\leq& e^{C\eta}\big( \| H_0 \|_{\dot{B}_{p,1}^{s}}+\|\nabla u
\|_{\tilde{L}_T^1(\dot{B}_{p,1}^{s})}
  +C \| \xi \|_{\tilde{L}_T^\infty (\dot{B}_{p,1}^{s})}
  \|\nabla u \|_{\tilde{L}_T^1(\dot{B}_{p,1}^{s})}\big)\nonumber\\
&\leq& e^{C\eta}\big( \| H_0 \|_{\dot{B}_{p,1}^{s}}+\eta +C\eta C_0 E_0 \big).
\end{eqnarray}
Taking  advantage of  Proposition \ref{p2.2} and Lemma \ref{l-2.2},
we obtain
\begin{eqnarray}\label{bl-E3.4}
\renewcommand{\arraystretch}{1.25}
\|\upsilon \|_{\tilde{L}_T^\infty(\dot{B}_{p,1}^{s-1})} &\leq& C
\big( \| \upsilon_0
\|_{\dot{B}_{p,1}^{s-1}}+\|G\|_{\tilde{L}_T^1(\dot{B}_{p,1}^{s-1})}\nonumber\\
&&+\|\nabla P\|_{\tilde{L}_T^1(\dot{B}_{p,1}^{s-1})}
 +\|a\nabla P\|_{\tilde{L}_T^1(\dot{B}_{p,1}^{s-1})}\big)\nonumber\\
&\leq& C \big( \| \upsilon_0 \|_{\dot{B}_{p,1}^{s-1}}
+\|G\|_{\tilde{L}_T^1(\dot{B}_{p,1}^{s-1})}\nonumber\\
&& +(1+\|a\|_{\tilde{L}_T^\infty(\dot{B}_{p,1}^{s})})
\|\nabla P\|_{\tilde{L}_T^1(\dot{B}_{p,1}^{s-1})}\big).
\end{eqnarray}
 Applying div on both sides of (\ref{bl-E3.2})$_2$ and by the incompressible condition div $\upsilon=0$, we have
$$
\text{div} \big( (a+1)\nabla P \big) = \text{div} G.
$$
Since $\dot{B}_{p,1}^{N/p}\hookrightarrow
\dot{B}_{p,\infty}^{N/p}\cap L^\infty$, therefore  $1-R\leq
1-\|a\|_{L^\infty (\dot{B}_{p,1}^s)}\leq 1+a\leq 1+R$.  Choosing
$R<\frac{1}{2}$, then using Lemma \ref{12.3} to the above equation,
we get
\begin{eqnarray}\label{bl-E3.5}
\renewcommand{\arraystretch}{1.25}
\|\nabla P\|_{\tilde{L}_T^1(\dot{B}_{p,1}^{s-1})} &\leq& C \big(
\|G\|_{\tilde{L}_T^1(\dot{B}_{p,1}^{s-1})} +
  \|a\|_{\tilde{L}_T^\infty (\dot{B}_{p,1}^{s})}\|\nabla P\|_{\tilde{L}_T^1(\dot{B}_{p,1}^{s-1})}\big)\nonumber\\
&\leq& C\|G\|_{\tilde{L}_T^1(\dot{B}_{p,1}^{s-1})} +CR \|\nabla P\|_{\tilde{L}_T^1(\dot{B}_{p,1}^{s-1})},
\end{eqnarray}
here we have used Lemma \ref{l-2.2}. Choosing $R<1/2$ so small that
$CR<1$, (\ref{bl-E3.5}) becomes
\begin{equation}\label{p}
\|\nabla P\|_{\tilde{L}_T^1(\dot{B}_{p,1}^{s-1})} \leq C
\|G\|_{\tilde{L}_T^1(\dot{B}_{p,1}^{s-1})}.
\end{equation}
Thus, combining (\ref{bl-E3.4}) and (\ref{p}) yield that
\begin{equation}\label{v}
\|\upsilon \|_{\tilde{L}_T^\infty(\dot{B}_{p,1}^{s-1})} \leq C \big[
\|\upsilon_0\|_{\dot{B}_{p,1}^{s-1}}+ (1+\|a\|_{\tilde{L}_T^\infty
(\dot{B}_{p,1}^{s})})\|G\|_{\tilde{L}_T^1(\dot{B}_{p,1}^{s-1})}
\big].
\end{equation}
By the definition of $G$ and Lemma \ref{l-2.2} we infer that
\begin{eqnarray}\label{g}
\renewcommand{\arraystretch}{1.25}
\|G\|_{\tilde{L}_T^1(\dot{B}_{p,1}^{s-1})}
 &\leq &C\|u\|^2_{\tilde{L}_T^2(\dot{B}_{p,1}^{s})} + C\|a\|_{\tilde{L}_T^\infty
          (\dot{B}_{p,1}^{s})} \|u\|_{\tilde{L}_T^1(\dot{B}_{p,1}^{s+1})}\nonumber\\
&& \quad \quad + CT\|\xi\|_{\tilde{L}_T^\infty
(\dot{B}_{p,1}^{s})}\big( 1+\|\xi\|_{\tilde{L}_T^\infty
(\dot{B}_{p,1}^{s})} \big).
\end{eqnarray}
Then  (\ref{p}) becomes
$$
\|\nabla P\|_{\tilde{L}_T^1(\dot{B}_{p,1}^{s-1})} \leq C \big[ R\eta
+\eta^2 + TC_0 E_0 (1+C_0 E_0)\big].
$$
Combining (\ref{bl-E3.3}) ,  (\ref{v}) and  (\ref{g}), we have
\begin{eqnarray*}
\renewcommand{\arraystretch}{1.25}
&&\|\upsilon\|_{\tilde{L}_T^\infty(\dot{B}_{p,1}^{s-1})}+
        \|H\|_{\tilde{L}_T^\infty(\dot{B}_{p,1}^{s})}\nonumber\\
&&\leq C e^{C\eta}\big[ E_0+ \eta(1+C_0 E_0)\big]
 + C(1+R)\big[ R\eta +\eta^2 + TC_0 E_0 (1+C_0 E_0)\big].
\end{eqnarray*}
By Proposition \ref{p2.2}, we can obtain
\begin{eqnarray*}
\renewcommand{\arraystretch}{1.25}
&&\|\upsilon\|_{\tilde{L}_T^1(\dot{B}_{p,1}^{s+1})}+
        \|\upsilon\|_{\tilde{L}_T^2 (\dot{B}_{p,1}^{s})}\\
&&\leq C \sum_{q\in \mathbb{Z}}2^{q(s-1)}\|\Delta_q u_0\|_{L^p}(1-e^{-c2^{2q}T})
     + C\| G -(a+1)\nabla P \|_{\tilde{L}_T^1(\dot{B}_{p,1}^{s-1})}\\
 &&\leq  C \sum_{q\in \mathbb{Z}}2^{q(s-1)}\|\Delta_q u_0\|_{L^p}(1-e^{-c2^{2q}T})\\
&&\quad + C(1+R)\big[ R\eta +\eta^2 + TC_0 E_0 (1+C_0 E_0)\big],
\end{eqnarray*}
here we have used (\ref{p}) and (\ref{g}).

Taking $C_0 =6C, \ R_0 \leq \frac{2}{3}R$, $\eta$  small such that
$e^{c\eta}\leq \frac{3}{2}$, $\eta (1+C_0 E_0)\leq E_0$ and choosing
$R<{1}/{2}$ small enough such that $(1+R)(R\eta +\eta^2)\leq
\min\{E_0,\frac{\eta}{4C}\}$. Next, we choose $T$ small such that
$CT(1+R) (1+C_0 E_0)\leq  \min \{\frac{1}{3} , \frac{\eta }{2C_0
E_0}\}$ and $C \sum_{q\in \mathbb{Z}}2^{q(s-1)}\|\Delta_q
u_0\|_{L^p}(1-e^{-c2^{2q}T})\leq \frac{\eta}{4}$. Then
\begin{eqnarray*}
&& \|\sigma \|_{\tilde{L}_T^\infty(\dot{B}_{p,1}^{\frac{N}{p}})}\leq
R,\quad \|\upsilon \|_{\tilde{L}_T^1(\dot{B}_{p,1}^{\frac{N}{p}+1})}
   +\|\upsilon \|_{\tilde{L}_T^2 (\dot{B}_{p,1}^{\frac{N}{p}})}\leq \eta,\\
&&\|\upsilon \|_{\tilde{L}_T^\infty(\dot{B}_{p,1}^{\frac{N}{p}-1})}
   +\|H \|_{\tilde{L}_T^\infty(\dot{B}_{p,1}^{\frac{N}{p}})}\leq C_0 E_0,
   \quad \|\nabla P\|_{\tilde{L}_T^1(\dot{B}_{p,1}^{\frac{N}{p}-1})}<\infty.
\end{eqnarray*}
Therefore, the functional $\Phi$ maps $\mathcal {A}$ to $\mathcal
{A}$.

Step 3. Time derivatives

The compactness of $\Phi$ will be supplied by the following lemma.

\begin{lem}\label{l3.1}
Denote $(\sigma,\upsilon, H):=\Phi(a,u,\xi)$, let $(a,u,\xi)$ be in
$\mathcal {A}$ with $T, \eta, C_0, R_0$ and $ R$ chosen according
to Step 2. Then $\sigma_t, H_t\in L_T^2 (\dot{B}_{p,1}^{s-1})$ and
$\upsilon_t\in L_T^{\frac{2}{1+\alpha}}
(\dot{B}_{p,1}^{s-1}+\dot{B}_{p,1}^{s-2+\alpha})$ for any $\alpha\in
[-1,1]$ such that $\alpha >\max(2-2N/p,2-N)$. Moreover, there exists
a constant $\overline{C}$ depending on $T,\eta,C_0,R_0, R$ and $E_0$
such that
$$
\|\sigma_t\|_{L_T^2 (\dot{B}_{p,1}^{s-1})} + \|H_t\|_{L_T^2
(\dot{B}_{p,1}^{s-1})}+ \|\upsilon_t\|_{L_T^{\frac{2}{1+\alpha}}
(\dot{B}_{p,1}^{s-1}+\dot{B}_{p,1}^{s-2+\alpha})}\leq \overline{C}.
$$
\end{lem}
The similar proof can be found in [3,12], we omit here.

Step 4.  The fixed point argument\\
We first introduce a  functional space
$$
Y_T^p :=\tilde{L}_T^\infty (\dot{B}_{p,1}^{s}) \times \big(
\tilde{L}_T^\infty (\dot{B}_{p,1}^{s-1})\cap \tilde{L}_T^2
(\dot{B}_{p,1}^{s}) \cap {\mathcal
{M}}_T(\dot{B}_{p,1}^{s+1})\big)^{N}
 \times \tilde{L}_T^\infty (\dot{B}_{p,1}^{s})^{N^2},
$$
where ${\mathcal {M}}_T(\dot{B}_{p,1}^{s+1})$ stands for the space
of bounded measures on $[0,T]$ with values in $\dot{B}_{p,1}^{s+1}$.
The space $Y_T^p$ endowed with the norm
$$
\|(a,u,\xi)\|_{Y_T^p} := \|(a,\xi)\|_{\tilde{L}_T^\infty
(\dot{B}_{p,1}^{s})}+\|u\|_{\tilde{L}_T^\infty
(\dot{B}_{p,1}^{s-1})\cap \tilde{L}_T^2
(\dot{B}_{p,1}^{s})}+\int_0^T d \|u\|_{\tilde{L}_T^\infty
(\dot{B}_{p,1}^{s+1})}
$$
is a Banach space. Furthermore, $Y_T^p$ is the dual space of
$$
X_T^p :=\tilde{L}_T^1 (\dot{B}_{p',\infty}^{-s}) \times \big(
\tilde{L}_T^1 (\dot{B}_{p',\infty}^{1-s})+ \tilde{L}_T^2
(\dot{B}_{p',\infty}^{-s}) + {\mathcal
{C}}([0,T];\dot{B}_{p',\infty}^{-s-1})\big)^{N}
 \times \tilde{L}_T^1 (\dot{B}_{p',\infty}^{-s})^{N^2},
$$
where $\tilde{L}_T^1 (B_{q,\infty}^{k})$ stands for the completion
of ${\mathscr{S}}([0,T];{\mathbb{R}}^N)$ under the norm of
$\tilde{L}_T^1 (B_{q,\infty}^{k})$, and that $X_T^p$ is a separable
Banach space.

Let $\overline{C}$ be as in Lemma \ref{l3.1}. We denote
$$
{\mathcal {D}}:=\big{\{}  (a,u,\xi)\in \mathcal {A};
 \|(a_t,\xi_t)\|_{L_T^2 (\dot{B}_{p,1}^{s-1})}+ \|u_t\|_{L_T^{\frac{4}{3}}
(\dot{B}_{p,1}^{s-1}+\dot{B}_{p,1}^{s-3/2})}\leq \overline{C}
\big{\}}.
$$
Since $Y_T^p$ is the dual  space of a Banach space, we gather that
$Y_T^p$ endowed with the weak star topology is a convex topological
linear space. Obviously, $\mathcal {A}$ is a convex subset of
$Y_T^p$. From Lemma \ref{l3.1}, we know that $\Phi(\mathcal
{A})\subset \mathcal {D}$. Since $\mathcal {D} \in \mathcal {A}$, it
is clear that  $\Phi$ is a self-mapping of $\mathcal {D}$. Just like
the proof in [3,12], we can obtain the continuity and
compactness of $\Phi$. Then Hukuhara's theorem ensures that the map
$\Phi$ has a fixed point $(\sigma, \upsilon, H)\in \mathcal {A}$,
which is a solution of (\ref{bl-E3.1}).

Furthermore, we can check that the right hand sides of
(\ref{bl-E3.1})$_1$ , (\ref{bl-E3.1})$_2$ and (\ref{bl-E3.1})$_3$
belongs to $L_T^1 (\dot{B}_{p,1}^{s})$, then Proposition \ref{p2.1}
and Proposition \ref{p2.2} insure $\sigma, H , \upsilon \in
{\mathcal{C}}([0,T];\dot{B}_{p,1}^{s})$.

Step 5. The proof of uniqueness

We assume that $(\sigma^i, \upsilon^i, H^i,\nabla P^i)\in
\tilde{E}_T^p \times \tilde{L}_T^1 (\dot{B}_{p,1}^{\frac{N}{p}-1})^N
\ (i=1,2;  p\in[1,N])$ are two solutions  of (\ref{bl-E3.1}) with
the same initial data. Set
$$
(\delta \sigma,\delta\upsilon,\delta H,\nabla \delta P)=(\sigma^1
-\sigma^2,\upsilon^1 -\upsilon^2,H^1 -H^2,\nabla P^1 -\nabla P^2).
$$
Then $(\delta \sigma,\delta\upsilon,\delta H,\nabla \delta P)$
satisfies
\begin{equation}
    \left\{
     \renewcommand{\arraystretch}{1.25}
    \begin{array}{l}
        \delta\sigma_t +  \upsilon^2 \cdot \nabla \delta\sigma=\delta G_1, \\
 \delta\upsilon_t  -\mu \Delta \delta \upsilon=\delta G_2,\\
 \delta H_t +\upsilon^2 \cdot \nabla \delta H=\delta G_3,
    \end{array}
    \right.\label{bl-E3.6}
    \end{equation}
 where
 \begin{eqnarray*}
\delta G_1 &:=& -\delta \upsilon \nabla \sigma^1,\\
\delta G_2 &:=& \mu\sigma^1 \Delta\upsilon^1-\upsilon^1 \cdot \nabla
\upsilon^1  -(\sigma^1+1)\nabla P^1+\text{div}H^1 + \text{div}(H^1 H^{1T})\\
&&\quad  -\mu\sigma^2 \Delta\upsilon^2 +\upsilon^2 \cdot\nabla
\upsilon^2  -\text{div}H^2 - \text{div}(H^2 H^{2T})+(\sigma^2+1)\nabla P^2,\\
\delta G_3 &:=& \nabla \delta\upsilon +\nabla \upsilon^1 \cdot H^1
-\nabla \upsilon^2 \cdot H^2 -\delta\upsilon \cdot \nabla H^1.
 \end{eqnarray*}
In what follows, we set $V^i (t):=\int_0^t \|\upsilon^i (\tau)
\|_{\dot{B}_{p,1}^{s+1}}d \tau$ for $i=1,2$, and denote by $A_T$ a
constant depending on
$\|(\sigma^1,\sigma^2)\|_{\tilde{L}_T^{\infty}(\dot{B}_{p,1}^{s})}$
and $\|(H^1,H^2)\|_{\tilde{L}_T^{\infty}(\dot{B}_{p,1}^{s})}$. Since
for any  $p\in [1,N],\  \tilde{E}_T^p \subseteq \tilde{E}_T^{N}$. So
we take $p=N$ in the sequel. Applying the Proposition \ref{p2.1} and
Lemma \ref{l-2.1}, we get that for any $t\in [0,T]$,
\begin{eqnarray*}
\|\delta \sigma(t) \|_{\dot{B}_{N,\infty}^0} &\leq&
 e^{CV^2 (T)}\int_0^t \|\delta\upsilon\|_{\dot{B}_{N,1}^{1}}\|\sigma^1\|_{\dot{B}_{N,\infty}^{1}}d\tau,\\
  \|\delta H(t)\|_{\dot{B}_{N,\infty}^0}
&\leq& e^{CV^2 (T)}\int_0^t \Big(
\|\upsilon\|_{\dot{B}_{N,1}^{1}}(1+\|H^1\|_{\dot{B}_{N,\infty}^{1}})\\
&&\quad +\|\upsilon^2\|_{\dot{B}_{N,1}^{2}}\|\delta
 H\|_{\dot{B}_{N,\infty}^{0}}\Big) d\tau.
\end{eqnarray*}
By the above estimates and Gronwall's inequality, we obtain
\begin{equation}\label{guji}
\|\delta \sigma \|_{\dot{B}_{N,\infty}^{0}}+\|\delta H
\|_{\dot{B}_{N,\infty}^{0}}\leq e^{CV^2 (T)}\int_0^t
\|\delta\upsilon \|_{\dot{B}_{N,1}^{1}}(1+\|\sigma^1
\|_{\dot{B}_{N,1}^{1}}+\| H^1 \|_{\dot{B}_{N,1}^{1}})d\tau.
\end{equation}
Using Proposition \ref{p2.2}, we know
\begin{equation}\label{v1}
\|\delta\upsilon
\|_{\tilde{L}_t^1(\dot{B}_{N,\infty}^{1})}+\|\delta\upsilon
\|_{\tilde{L}_t^2 (\dot{B}_{N,\infty}^{0})} \leq \|\delta G_2
\|_{\tilde{L}_t^1(\dot{B}_{N,\infty}^{-1})}.
\end{equation}
Since $ \text{div}\  (\delta G_2)=\text{div}\ (\delta \upsilon_t
-\mu\Delta\delta \upsilon)=0, $ thus
$$
\text{div}\ ((\sigma^1 +1)\nabla \delta P)=\text{div}\ F,
$$
where
\begin{eqnarray*}
0&<& c_1\leq \sigma^1 +1 \leq c_2,\\
 F&=& \mu \delta \sigma \Delta
\upsilon^1 +\mu\sigma^2 \Delta \delta\upsilon -\upsilon^1
\nabla\delta\upsilon -\delta\upsilon\nabla \upsilon^2
-\delta\sigma\nabla P^2\\
&&\quad + \text{div}\ \delta H +\text{div}\ (H^1
H^{1T}-H^2 H^{2T}).
\end{eqnarray*}
By using the estimate (\ref{les}) to the above equation, we have
\begin{equation}\label{dp}
\|\nabla \delta P\|_{\dot{B}_{N,\infty}^{-1}}\leq C \|
F\|_{\dot{B}_{N,\infty}^{-1}}+C\| \sigma^1\|_{\dot{B}_{N,1}^{1}}\|
\nabla \delta P\|_{\dot{B}_{N,\infty}^{-1}}.
\end{equation}
Choosing $C\| \sigma^1\|_{\dot{B}_{N,1}^{1}}\leq \frac{1}{2}$, by
the definition of $F$, we obtain
\begin{eqnarray*}
\|\delta G_2 \|_{\dot{B}_{N,\infty}^{-1}}
 &\leq& C \|\delta\upsilon \|_{\dot{B}_{N,\infty}^{0}}(\|\upsilon^1\|_{\dot{B}_{N,1}^{1}}
      +\|\upsilon^2\|_{\dot{B}_{N,1}^{1}})
    +C\|\sigma^2\|_{\dot{B}_{N,1}^{1}}\|\delta\upsilon\|_{\dot{B}_{N,\infty}^{1}}\\
    &&\quad + \|\delta H\|_{\dot{B}_{N,\infty}^{0}}
 +C\|\delta\sigma\|_{\dot{B}_{N,\infty}^{1}} (\|\upsilon^1\|_{\dot{B}_{N,1}^{1}}
 +\|\nabla P^2\|_{\dot{B}_{N,1}^{-1}})\\
 &&\quad +\|\text{div}(H^1
   H^{1T}-H^2 H^{2T})\|_{\dot{B}_{N,\infty}^{-1}}
   \end{eqnarray*}
We can take $\overline{T}\leq T$ small enough such that for any
$t\in [0,\overline{T}]$,
$$
\|\nabla P\|_{\tilde{L}_t^1(\dot{B}_{N,1}^{-1})}+
\|(\upsilon^1,\upsilon^2)\|_{\tilde{L}_t^1(\dot{B}_{N,1}^{2})\cap
\tilde{L}_t^2(\dot{B}_{N,1}^{1})}+
\|(\sigma^1,\sigma^2)\|_{\tilde{L}_t^\infty(\dot{B}_{N,1}^{1})}\ll
1.
$$
Therefore, (\ref{v1}) becomes
\begin{eqnarray}\label{v2}
\|\delta\upsilon \|_{\tilde{L}_t^1(\dot{B}_{N,\infty}^{1})} &\leq& A_T
\int_0^t \Big(\|\delta \sigma \|_{\dot{B}_{N,\infty}^{0}}(1+\|
\upsilon^1\|_{\dot{B}_{N,1}^{1}} +\| \nabla
P^2\|_{\dot{B}_{N,1}^{0}})\nonumber\\
&&\quad+\| \delta
H\|_{\dot{B}_{N,\infty}^{0}}\Big) d\tau.
\end{eqnarray}
Using (\ref{e2.5}), we obtain
$$
\|\delta\upsilon \|_{\tilde{L}_t^1(\dot{B}_{N,1}^{1})} \leq C
\|\delta\upsilon \|_{\tilde{L}_t^1(\dot{B}_{N,\infty}^{1})} \log
 \big( e+ C_T \|\delta\upsilon \|^{-1}_{\tilde{L}_t^1 (\dot{B}_{N,\infty}^{1})}\big),
$$
where $C_T =\|\delta\upsilon
\|_{\tilde{L}_t^1(\dot{B}_{N,\infty}^{0})}
 +\|\delta\upsilon \|_{\tilde{L}_t^1 (\dot{B}_{N,\infty}^{2})}$.
It yields that for any $t\in [0,\overline{T}]$,
\begin{eqnarray}\label{v2}
\|\delta\upsilon \|_{\tilde{L}_t^1(\dot{B}_{N,\infty}^{1})}
& \leq& A_T
\int_0^t \Big(\|\delta \upsilon
\|_{\tilde{L}_{\tau}^1(\dot{B}_{N,\infty}^{1})}
 (1+\|\upsilon^1\|_{\dot{B}_{N,1}^{1}} +\| \nabla
P^2\|_{\dot{B}_{N,1}^{-1}})\nonumber\\
&&\log(e+ C_T \|\delta\upsilon
\|^{-1}_{\tilde{L}_\tau^1 (\dot{B}_{N,\infty}^{1})})\Big) d\tau.
\end{eqnarray}
Since $\int_0^t \| \nabla P^2\|_{\dot{B}_{N,1}^{-1}} \ll 1$ for any
$t\in [0,\overline{T}]$,  $1+\|
\upsilon^1\|_{\dot{B}_{2N,1}^{{3}/{2}}}$ is integrable on $[0,T]$,
and
$$
\int_0^1 \frac{1}{r \log (e+ C_T r^{-1})}dr=+\infty,
$$
from Osgood lemma we know $(\delta\sigma,\delta\upsilon,\delta
H,\nabla\delta P)=(0,0,0,0)$ on $[0,\overline{T}]$. Using Continuity
argument, we can prove $(\sigma^1,\upsilon^1,H^1,\nabla
P^1)=(\sigma^2,\upsilon^2,H^2,\nabla P^2)$ on $[0,T]$.

Step 6. Global existence with small initial data

Note that $\|\cdot\|_{\tilde{B}_{\mu}^{s,\infty}}\approx
\|\cdot\|_{B^{s-1}\cap B^s}$  and
$\|\cdot\|_{\tilde{B}_{\mu}^{s,2}}=\|\cdot\|_{B^{s}}$. The above
steps ensure that there exists an positive time $T_1$ and a unique
solution $(\sigma,\upsilon,H)\in {\mathcal{H}}_T^{\frac{N}{2}},\
T\leq T_1$ such that
\begin{equation}\label{h}
\|(\sigma,H,\upsilon)\|_{{\mathcal {H}}_{T}^{\frac{N}{2}}}\leq C_1
(\|\sigma_0\|_{\tilde{B}_{\mu}^{\frac{N}{2},\infty}}+\|H_0\|_{\tilde{B}_{\mu}^{\frac{N}{2},\infty}}
+\|\upsilon_0\|_{B^{\frac{N}{2}-1}}),
\end{equation}
where $C_1>1$, and
\begin{eqnarray*}
 \|(\sigma,\upsilon,H)\|_{{\mathcal
{H}}_{T}^{s}}&:=&
\|(\sigma,H)(t)\|_{L_T^{\infty}(\tilde{B}_{\mu}^{s,\infty})}+
\|\upsilon(t)\|_{L_T^{\infty}(B^{s-1})}\\
&& +\mu\big(
\|(\sigma,H)(t)\|_{L_T^{1}(\tilde{B}_{\mu}^{s,1})}+\|\upsilon(t)\|_{L_T^{1}(B^{s+1})}\big).
\end{eqnarray*}
We define $d_{ij}=-\Lambda^{-1}\nabla_j \upsilon^i$,
 then $\upsilon^i=\Lambda^{-1}\nabla_j d^{ij}$. By applying $-\Lambda^{-1}\nabla_j$ to
 (\ref{bl-E3.1})$_2$, we get
\begin{equation}\label{4.1}
\partial_t d^{ij} -\mu\Delta d^{ij}+\Lambda^{-1}(\nabla_j \nabla_k
H^{ik})=\Lambda^{-1}\nabla_j (\upsilon\cdot\nabla\upsilon^i
-H^{lk}\nabla_l H^{ik}).
\end{equation}
Equation (\ref{bl-E1.4}) implies
\begin{equation}\label{4.2}
\Lambda^{-1}(\nabla_j \nabla_k H^{ik})=-\Lambda
H^{ij}-\Lambda^{-1}\nabla_k (H^{lj}\nabla_l H^{ik} -H^{lk}\nabla_l
H^{ij}).
\end{equation}
By plugging  (\ref{4.2}) into (\ref{4.1}), from (\ref{bl-E3.1}) we
 get
\begin{equation}
    \left\{
     \renewcommand{\arraystretch}{1.25}
    \begin{array}{l}
        \sigma_t +  \upsilon \cdot \nabla \sigma=0, \\
 \partial_t d^{ij} +\upsilon \cdot\nabla d^{ij}-\mu \Delta d^{ij} -\Lambda H^{ij}=G,\\
\partial_t H^{ij} +\upsilon \cdot\nabla H^{ij} +\Lambda d^{ij}=F,
    \end{array}
    \right.\label{e4.2}
    \end{equation}
    where
\begin{eqnarray*}
 \renewcommand{\arraystretch}{1.25}
F&=& \nabla_k \upsilon^i H^{kj}\\
G&=&\upsilon \cdot\nabla(-\Lambda^{-1}\nabla_j
          \upsilon^i)+\Lambda^{-1}\nabla_j [\upsilon\cdot\nabla\upsilon^i
       +(\sigma+1)\nabla_i P\\
       && \quad  -\mu \sigma\Delta \upsilon^i -H^{jk}\nabla_j H^{ik}]
+\Lambda^{-1}\nabla_k (H^{lj} \nabla_l
      H^{ik}-H^{lk}\nabla_l H^{ij}).
\end{eqnarray*}
We shall prove that there exists a constant $ M\geq C_1$ such that,
if the Besov norm of initial data $ \alpha:=
\|\sigma_0\|_{\tilde{B}_{\mu}^{s,\infty}}+ \|\upsilon_0\|_{B^{s-1}}+
\|H_0\|_{\tilde{B}_{\mu}^{s,\infty}} $ is small enough, then the
solution of (\ref{bl-E3.1}) satisfies
\begin{equation}\label{jx}
\|(\sigma,\upsilon,H)\|_{{\mathcal {H}}_T^{N/2}}\leq M\alpha
\end{equation}
 for any $T\in [0,+\infty)$.

 Assume (\ref{jx}) holds for $t\in (0,\tilde{T}]$. (\ref{h})
ensures that $\tilde{T}$ is bounded from blow by $T_1$ and there
exists a constant $T_2>0$ such that
$\|(\sigma,\upsilon,H)\|_{{\mathcal {H}}_T^{N/2}}\leq M\alpha$ for
any  $T\in [0,\tilde{T}+T_2)$. Since
$$
\partial_{i}((\sigma+1) \partial_i P) =\partial_i \big( -\upsilon\cdot\nabla \upsilon^i +\mu\sigma\Delta \upsilon^i
+ H^{jk} \partial_j H^{ik}\big),\quad \text{with}\ c_1\leq \sigma+1
\leq c_2,
$$
using Lemma \ref{12.3}, we obtain
\begin{eqnarray*}
 \renewcommand{\arraystretch}{1.25}
\|\nabla P\|_{L_T^1 (B^{N/2 -1})}&\leq& C \|\upsilon\|^2_{L_T^2
             (B^{N/2 })}+C\mu\|\sigma\|_{L_T^\infty (B^{N/2 })}
                 \|\upsilon\|_{L_T^1 (B^{N/2 +1})}\\
&&\quad+ C \|H\|^2_{L_T^2 (B^{N/2 })} + C \|\sigma\|_{L_T^\infty
                    (B^{N/2})}\|\nabla P\|_{L_T^1 (B^{N/2 -1})}\\
&\leq& C(C_1 M \alpha)^2+ CC_1 M \alpha\|\nabla P\|_{L_T^1 (B^{N/2
-1})},
\end{eqnarray*}
choose $\alpha$ so small that $CC_1 M \alpha \leq 1/2$, we have
\begin{equation}\label{xp}
\|\nabla P\|_{L_T^1 (B^{N/2 -1})}\leq C(C_1 M \alpha)^2.
\end{equation}
Now, we estimate $\|F\|_{L_T^1
(\tilde{B}_{\mu}^{\frac{N}{2},\infty})}$, since
\begin{equation}\label{f}
\|F\|_{L_T^1 (\tilde{B}_{\mu}^{\frac{N}{2},\infty})}\leq
C\|H\|_{L_T^\infty (\tilde{B}_{\mu}^{\frac{N}{2},\infty})} \|\nabla
\upsilon\|_{L_T^1 (B^{\frac{N}{2}})}\leq C(C_1 M \alpha)^2.
\end{equation}
For the term of $G$, using (\ref{xp}) we have
\begin{eqnarray*}
 \|G\|_{L_T^1(B^{\frac{N}{2}-1})} &\leq& C\|\sigma\|_{L_T^\infty
(B^{\frac{N}{2}})}\big(\|\nabla P\|_{L_T^1(B^{\frac{N}{2}-1})}\\
&& +C \|\upsilon\|^2_{L_T^2
(B^{\frac{N}{2}})}+\|\upsilon\|_{L_T^1(B^{\frac{N}{2}+1})}\big)
 + C\|H\|^2_{L_T^1(B^{\frac{N}{2}})}\\
  &\leq& C(C_1 M \alpha)^2.
\end{eqnarray*}
For the equation (\ref{bl-E3.1}), according to Proposition
\ref{p4.1}, we have
\begin{eqnarray*}
 \|(\sigma,\upsilon,H)\|_{{\mathcal {H}}_T^{N/2}}&\leq& C
 e^{C\|\upsilon\|_{L_T^1 (B^{\frac{N}{2}+1})}}\Big( \|\sigma_0\|_{\tilde{B}_{\mu}^{\frac{N}{2},\infty}}+
 \|\upsilon_0\|_{B^{\frac{N}{2}-1}}\\
&& +\|H_0\|_{\tilde{B}_{\mu}^{\frac{N}{2},\infty}}
  + \|F\|_{L_T^1 (\tilde{B}_{\mu}^{\frac{N}{2},\infty})} + \|G\|_{L_T^1
 (B^{\frac{N}{2}-1})}\Big),
\end{eqnarray*}
for $T\in [0,\tilde{T}+T_2]$. Thus, we obtain
\begin{equation}
 \|(\sigma,\upsilon,H,\nabla P)\|_{{\mathcal {H}}_T^{N/2}}\leq C
 e^{CC_1 M \alpha}(\alpha+C_1^2 M^2 \alpha^2),
\end{equation}
so choosing $M=\max (4C,C_1)$ and making the following assumptions:
$$
C_1^2 M^2 \alpha \leq 1,\quad e^{CC_1 M \alpha}\leq 2,\quad CC_1 M
\alpha <\frac{1}{2},
$$
then (\ref{jx}) holds for $T\in [0,\tilde{T}+T_2]$, hence for any
$T\in [0,+\infty)$, it is followed from a bootstrap argument. The
proof is complete.


\footnotesize
\section{References}
\begin{itemize}

\item[{[1]}]
J.Y. Chemin,   N. Lerner,  Flot de champs de vecteurs non
lipschitziens et ¨¦quations de Navier-Stokes, {\em J. Differential
Equations}, {\bf 121}, (1992) 314-328.

\item[{[2]}]
Y.M. Chen, P. Zhang,   The global existence of small solutions
 to the incompressible viscoelastic fluid system in 2 and 3 space
 dimensions,
 {\em Comm. Part. Diff. Equa }, {\bf 31}, (2006) 1793-1810.

\item[{[3]}]
R.  Danchin, Fourier Analysis Mathod for PDEs,  Lecture note.

\item[{[4]}]
R. Danchin,  Local theory in Critical spaces for Compressible viscous
and Heat-conductive gases, {\em Comm.  Part. Diff. Equa}, {\bf 26},
(2001) 1183-1233.

\item[{[5]}]
R. Danchin,  Density-dependent incompressible viscous fluids
in critical spaces, {\em Proc. Roy. Soc. Edinburgh Sect. A}, {\bf 133},
 (2003) 1311-1334.

\item[{[6]}]
R. Danchin,  Global existence in critical space for compressible
Navier-Stokes equations, {\em Invent. math}, {\bf 141}, (2000) 579-614.

\item[{[7]}]
K. Kunisch,  M.  Marduel,   Optimal control of non-isothermal
viscoelastic fluid flow, {\em J. Non-Newtonian Fluid Mechanics}, {\bf 88},
 (2000) 261-301.

\item[{[8]}]
Z. Lei,    Y. Zhou, Global existence of classical solutions for the
two-dimensional Oldroyd model via the incompressible limit, {\em SIAM J.
Math. Anal.}, {\bf 37}, (2005) 797-814.

\item[{[9]}]
Z. Lei,   C. Liu,  Y. Zhou,   Global solutions of incompressible
viscoelastic fluids,
{\em Arch. Ration. Mech. Anal.}, {\bf 188}, (2008) 317-398.

\item[{[10]}]
F. Lin,    P. Zhang,  On the initial-boundery value problem of the incompressible viscolastic fluid
 system,
{\em Comm. Pure Appl. Math}, {\bf 61}, (2008) 0539-0558.

\item[{[11]}]
 J. Qian, Well-posedness in critical spaces for incompressible viscoelastic
 fluid system, {\em Nonlinear Anal}, {\bf 72}(6),(2010) 3222-3234.
\item[{[12]}]
  J. Qian,  Z.  Zhang,  On the well-posedness in vritical spaces for
the compressible viscoelastic fluids, Inpress.

\item[{[13]}]
X. Hu, D. Wang,  Global strong solution to the density-dependent incompressible viscoelastic fluids,
Arxiv: 0905.0663v1.

\end{itemize}

Corresponding author: Huazhao Xie, \\
Department of Mathematics, Henan University of
Economics
and Law, Zhengzhou, 450002,  China\\
 E-mail:
hzh$_{-}$xie@yahoo.com.cn\\

Yunxia Fu,\\  Department of Basic Courses, PLA Commanding
Communicatioins Academy, Wuhan 430010,  China\\
E-mail: fuyunxia2005@yahoo.com.cn
\end{document}